\def\update{December 24, 2009}
\documentclass[twoside,8pt,b5paper]{article}

\usepackage{graphicx,amsmath,latexsym,amssymb,amsthm,fancyhdr}

\newcounter{Theorem}

\theoremstyle{plain}
\newtheorem{thm}[Theorem]{\bf Theorem}
\newtheorem{prop}[Theorem]{\bf Proposition}
\newtheorem{cor}[Theorem]{\bf Corollary}

\theoremstyle{remark}
\newtheorem{remark}[Theorem]{\bf Remark}

\theoremstyle{definition}

\numberwithin{Theorem}{section}

\newcommand{\R}{\ensuremath{\mathbb{R}}}
\newcommand{\Z}{\ensuremath{\mathbb{Z}}}
\newcommand{\Q}{\ensuremath{\mathbb{Q}}}
\newcommand{\N}{\ensuremath{\mathbb{N}}}

\begin{document}

\thispagestyle{empty}

\setcounter{page}{1}% ignore this

\vspace*{-1.5cm}

\noindent {\footnotesize{\textbf{Chamchuri Journal of Mathematics}}} \\
\noindent {\footnotesize{\textsc{Volume 1(2009) Number 1, \pageref{1st-page}--\pageref{last-page}}}} \\
\noindent {\footnotesize{\texttt{http://www.math.sc.chula.ac.th/cjm}}}
\hfill {\footnotesize{\texttt{updated: \update}}}
\\

%\vspace{-2.4cm}
%
%\begin{figure}[h]
%\begin{flushright}
%\includegraphics[width=2.5cm,height=1.5cm]{chamchuri.eps}
%\end{flushright}
%\end{figure}

\label{1st-page}

\renewcommand{\thefootnote}{\fnsymbol{footnote}}

\begin{center}
{\LARGE{\textbf{On the criteria for linear independence}}} \\[0.2cm]
{\LARGE{\textbf{of Nesterenko, Fischler and Zudilin}}} \\[0.4cm]
{\Large{Chantanasiri Amarisa}} 
%{\Large{Name1 Surname1{\footnote{\textit{The author is supported by xxx}}},
%        Name2 Surname2{\footnote{\textit{Corresponding author}}}}} \\[0.2cm] 
%{\Large{and Name3 Surname3}}
 \\[0.2cm]
\end{center}

\vspace{0.25cm}

%\begin{flushright}
%\begin{small}
%\textit{Received dd mm 2009} \\
%\textit{Revised dd mm 2009} \\
%\textit{Accepted dd mm 2009} \\
%\end{small}
%\end{flushright}

\renewcommand{\thefootnote}{\arabic{footnote}}

\bigskip

\noindent
\textbf{Abstract:\ } %Your abstract.
In 1985, Yu.~V.~Nesterenko produced a criterion for linear independence, which is a variant of Siegel's. While Siegel
uses upper bounds on full systems of forms, Nesterenko uses upper and lower bounds on sufficiently dense sequences of individual forms. The proof of Nesterenko's criterion was simplified by F.~Amoroso and P.~Colmez in 2003. More recently, S.~Fischler and W.~Zudilin produced a refinement, together with a much simpler proof. This new proof rests on a simple argument which we expand here. We get a new result, which contains Nesterenko's criterion, as well as criteria for algebraic independence.

\vspace{0.6cm}

\medskip

\bigskip

%%%%%%%%%%%%%%%%%%%%%%%%%%%%%%%%%%%%%%%%%%%%%%%%%%%%%%%%%%%%%%%%%%%%%%%%%%%%%%%%%%%%%%%%%%%%%%%%%%%%%%%%%%%%%%%%%%%
%%%%%%%%%%%%%%%%%%%%%%%%%%%%%%%%%%%%%%%%%%%%%%%%%%%%%%%%%%%%%%%%%%%%%%%%%%%%%%%%%%%%%%%%%%%%%%%%%%%%%%%%%%%%%%%%%%%
\section{Introduction}
%%%%%%%%%%%%%%%%%%%%%%%%%%%%%%%%%%%%%%%%%%%%%%%%%%%%%%%%%%%%%%%%%%%%%%%%%%%%%%%%%%%%%%%%%%%%%%%%%%%%%%%%%%%%%%%%%%%
%%%%%%%%%%%%%%%%%%%%%%%%%%%%%%%%%%%%%%%%%%%%%%%%%%%%%%%%%%%%%%%%%%%%%%%%%%%%%%%%%%%%%%%%%%%%%%%%%%%%%%%%%%%%%%%%%%%
%\subsection{Siegel's Method}
%%%%%%%%%%%%%%%%%%%%%%%%%%%%%%%%%%%%%%%%%%%%%%%%%%%%%%%%%%%%%%%%%%%%%%%%%%%%%%%%%%%%%%%%%%%%%%%%%%%%%%%%%%%%%%%%%%%

In his fundamental paper  \cite{Siegel1929} in 1929, C.L.~Siegel introduced the following result.

\begin{thm}[Siegel's linear independence criterion]\label{Thm:Siegel}
Let $\underline{\vartheta} = (\vartheta_1, \ldots, \vartheta_m) \in \R^{m}$. Assume that, for all  $\varepsilon > 0$, there exists a complete system of $m+1$  linearly independent linear  forms in $m+1$ variables 
\[
L_i = L_i(\underline{X}) = \displaystyle\sum_{j=0}^{m} b_{ij} X_j, \quad i=0,1,\ldots,m, \quad b_{ij} \in \Z,
\]
such that
 $$
 \max_{0 \le i \le m} |L_i(1, \underline{\vartheta})| \le \frac{\varepsilon}{A^{m-1}}
 \quad\hbox{ where }\quad A = \max_{0\le i,j\le m} |b_{ij}|.
 $$
Then the numbers $1, \vartheta_1, \ldots, \vartheta_m$ are linearly independent over $\Q$.
\end{thm}
%%%%%%%%%%%%%%%%%%%%%%%%%%%%%%%%%%%%%%%%%%%%%%%%%%%%%%%%%%%%%%%%%%%%%%%%%%%%%%%%%%%%%

This result is discussed in  \cite{MR1603608} Chap.~2 \S~1.4. We reproduce the proof here. As pointed out by Siegel, this argument  yields not only linear independence results, but also quantitative refinements (measures of linear independence) which we do not consider here. Studying only qualitative aspects of the subject enables us to simplify the situation. 

%%%%%%%%%%%%%%%%%%%%%%%%%%%%%%
\begin{proof}
Let 
$$
L(\underline{X}) = a_0 X_0 + \cdots + a_m X_m, \quad a_j \in \Z,
$$
be  a non--zero linear form with integer coefficients in $m+1$ variables.
The aim is to prove, under the assumptions of Theorem \ref{Thm:Siegel},  $L(1,\underline{\vartheta}) \neq 0$. Set
$$
H=\max_{0 \le j \le m}  |a_j|.
$$
Let $\varepsilon$ be a positive real number $< 1/(m! \cdot mH)$.
Among the forms $\{L_0, \ldots, L_m\}$ satisfying the assumptions of Theorem \ref{Thm:Siegel} for this value of $\varepsilon$, there exist $m$ of them, say $L_{k_1}, \ldots, L_{k_m}$, which along with $L$ make up a complete system of linearly independent forms.
Denote by $\Delta$ the determinant of the coefficient matrix of the system of linear forms $L,L_{k_1}, \ldots, L_{k_m}$ and, for $0\le i,j\le m$,  by $\Delta_{i,j}$ the $(i,j)$-minor of this matrix.
Then 
\[
\Delta = L(1,\underline{\vartheta}) \cdot \Delta_{0,0} + \displaystyle\sum_{i=1}^{m} L_{k_i}(1,\underline{\vartheta}) \cdot \Delta_{i,0}.
\]
Since $\Delta \in \Z$ and $\Delta \neq 0$, we have $|\Delta| \ge 1$. One easily estimates, for $0\le j\le m$, 
$$
|\Delta_{0,j}|\le m! A^m \quad\hbox{ and }\quad
\max_{1\le i\le m} |\Delta_{i,j}|\le m! H A^{m-1}.
$$
It follows that
\begin{align*}
(m!)^{-1} &\le |L(1,\underline{\vartheta}))| \cdot A^m + \displaystyle\sum_{i=1}^{m} |L_{k_i}(1,\underline{\vartheta})| \cdot H A^{m-1}  \\
&\le |L(1,\underline{\vartheta})| \cdot A^m + \varepsilon \cdot mH   \\
&<  |L(1,\underline{\vartheta})| \cdot A^m + (m!)^{-1}.
\end{align*}
Thus $L(1,\underline{\vartheta}) \neq 0$.
\end{proof}
%%%%%%%%%%%%%%%%%%%%%%%%%%%%%%%%%%%%%%%%%%%%%%%%%%%%%%%%%%%%%%%%%%%%%%%%%%%%%%%%%%%%%%%%%%%%%%%%%%%%%%%%%%%%%%%%%%%
%%%%%%%%%%%%%%%%%%%%%%%%%%%%%%%%%%%%%%%%%%%%%%%%%%%%%%%%%%%%%%%%%%%%%%%%%%%%%%%%%%%%%%%%%%%%%%%%%%%%%%%%%%%%%%%%%%%
%\subsection{Nesterenko's criterion}
%%%%%%%%%%%%%%%%%%%%%%%%%%%%%%%%%%%%%%%%%%%%%%%%%%%%%%%%%%%%%%%%%%%%%%%%%%%%%%%%%%%%%%%%%%%%%%%%%%%%%%%%%%%%%%%%%%%

Siegel used this approach to prove transcendence results (linear independence of $1,\vartheta,\ldots,\vartheta^m,\ldots$). 

In 1985, Yu.~V.~Nesterenko  (see the corollary of the theorem in \cite{MR783238}) introduced a different type of criterion, involving a sequence of linear forms (and not a sequence of complete systems of linear forms); for each of them, he requires not only an upper bound, but also a lower bound.  

\begin{thm}[Nesterenko's linear independence criterion]   \label{Nest}
Let $c_1, c_2,\tau_1,\tau_2$ be positive real numbers and $\sigma(n)$ a non--decreasing positive  function such that  
\[
\lim_{n \rightarrow \infty} \sigma(n) = \infty
\quad\hbox{ and } \quad  
\overline{\lim_{n \rightarrow \infty}}  ~\frac{\sigma(n+1)}{\sigma(n)} = 1.
\] 
Let $\underline{\vartheta} = (\vartheta_1, \ldots, \vartheta_m) \in \R^{m}$. 
Assume that, for all sufficiently large integers $n$, there exists a linear form with integer coefficients in $m+1$ variables
\[
L_n(\underline{X})=\ell_{0n}X_0+\ell_{1n}X_1+\cdots+\ell_{m n}X_{m},
\]
which satisfies the conditions
\[
\sum_{i=0}^{m} |\ell_{in}| \le e^{\sigma(n)}
\quad\hbox{ and } \quad
c_1 e^{-\tau_1 \sigma(n)}\le |L_n(1, \underline{\vartheta})|\le c_2 e^{-\tau_2 \sigma(n)}.
\]
Then $\dim_{\Q} (\Q + \Q \vartheta_1 + \cdots + \Q \vartheta_m) \ge (1 + \tau_1) / (1+ \tau_1 - \tau_2)$.
\end{thm}
%%%%%%%%%%%%%%%%%%%%%%%%%%%%%%%%%%%%%%%%%%%%%%%%%%%%%%%%%%%%%%%%%%%%%%%%%%%%%%%%%%%%%%%%%%%%%%%%%%%%%%%%%%%%%%%%%%%
%%%%%%%%%%%%%%%%%%%%%%%%%%%%%%%%%%%%%%%%%%%%%%%%%%%%%%%%%%%%%%%%%%%%%%%%%%%%%%%%%%%%%%%%%%%%%%%%%%%%%%%%%%%%%%%%%%%

The original proof by Nesterenko was rather involved; it has been simplified by F.~Amoroso \cite{Amoroso} and this simplification was revisited (and translated from Italian to French) by P.~Colmez \cite{Colmez}. Recently, S.~Fischler and W.~Zudilin \cite{FischlerZudilin} obtained a refinement of Nesterenko's Criterion. Their proof of this refinement is much easier than the previous proofs.
%and it is even simpler if we restrict to the Theorem~\ref{Nest} without introducing gcd's.
We develop this refinement in the second section, also we show that this result contains Nesterenko's Criterion. Then we deduce criteria for algebraic independence and transcendence criteria in the third and the fourth sections, respectively.

\section{Main theorem and its proof}
%%%%%%%%%%%%%%%%%%%%%%%%%%%%%%%%%%%%%%%%%%%%%%%%%%%%%%%%%%%%%%%%%%%%%%%%%%%%%%%%%%%%%%%%%%%%%%%%%%%%%%%%%%%%%%%%%%%
%%%%%%%%%%%%%%%%%%%%%%%%%%%%%%%%%%%%%%%%%%%%%%%%%%%%%%%%%%%%%%%%%%%%%%%%%%%%%%%%%%%%%%%%%%%%%%%%%%%%%%%%%%%%%%%%%%%

Here is our main result.

\begin{thm}\label{MainTh}
Let $\underline{\xi}=(\xi_i)_{i\ge 0}$ be a sequence of real numbers with $\xi_0=1$, 
$(r_n)_{n\ge 0}$ a non--decreasing sequence of positive integers, $(Q_n)_{n\ge 0}$, $(A_n)_{n\ge 0}$ and $(B_n)_{n\ge 0}$ sequences of positive real numbers such  that $\lim_{n\rightarrow\infty} A_n^{1/r_n}=\infty$ and, for all sufficiently large integers $n$,
\[
Q_n B_n \le Q_{n+1} B_{n+1}.
\]
For any integer $n\ge 0$, let 
$$
L_n(\underline{X})=\ell_{0n}X_0+\ell_{1n}X_1+\cdots+\ell_{r_n n}X_{r_n}
$$
be a linear form with integer coefficients in $r_n+1$ variables.
Assume that, for any sufficiently large integer $n$,
\[ 
\sum_{i=0}^{r_n} |\ell_{in}|\le Q_n,\quad
0<|L_n(\underline{\xi})|\le \frac{1}{A_n}
\quad\hbox{ and } \quad
\frac{|L_{n-1}(\underline{\xi})|}{|L_n(\underline{\xi})|}\le B_n.
\]
Then 
$$
A_n \le 2^{r_n+1}(B_n Q_n)^{r_n}
$$ 
for all sufficiently large integers $n$.
\end{thm}
%%%%%%%%%%%%%%%%%%%%%%%%%%%%%%%%%%%%%%%%%%%%%%%%%%%%%%%%%%%%%%%%%%%%%%%%%%%%%%%%%%%%%%%%%%%%%%%%%%%%%%%%%%%%%%%%%%%
\begin{proof}
For $n$ a sufficiently large integer, let $\mathsf{C}_n$ be the convex symmetric compact set in $\R^{1+r_n}$ defined by
\begin{equation}  \label{1}
|x_0|\le \frac{1}{2|L_n(\underline{\xi})|},\quad
|x_0\xi_i-x_i|\le |2 L_n(\underline{\xi})|^{1/r_n} \quad (1\le i\le r_n).
\end{equation}
The volume of $\mathsf{C}_n$ is $  2^{r_n+1}$. 
From Minkowski's Convex Body Theorem (Theorem 2B of Chapter II, \cite{MR568710}), there is a non-zero integer point in $\mathsf{C}_n$. We fix such a $\underline{x}(n)=(x_0(n),\ldots,x_{r_n}(n))\in\Z^{r_n+1}\setminus\{0\}$. Since $A_n^{1/r_n} \rightarrow\infty$ as $n \rightarrow\infty$, we have $x_0(n)\ne 0$ for $n$ sufficiently large. Indeed, if we had $x_0(n)=0$ with $A_n > 2$, we would have $|x_i(n)| \le (2/A_n)^{1/r_n} < 1$, and then $x_i(n)=0$, for each $i=1,...,r_n$, which contradicts the choice of $x(n)$. 
 
From the assumptions and the estimate  $|2L_n (\underline{\xi})|^{1/r_n} \le (2/ A_n)^{1/r_n}$ for $n$ sufficiently large, it follows that the sequence  $|2L_n (\underline{\xi})|^{1/r_n}$ tends to zero as $n \rightarrow\infty$.
Since $r_n \ge 1$, $|L_n (\underline{\xi})|$ also tends to zero.

Now if the sequence $|x_0(n)|$ did not tend to infinity, it admits a bounded subsequence, so a constant subsequence, and we would have $x_0(n)=y$, with $y \in \Z$, $y \ne 0$, for all $n$ belonging to an infinite subset $A$ of $\N$. For all integers $i \le \sup r_n$, we would have $\lim_{n \in A, n \to \infty} (y \xi_i-x_i(n)) = 0$, therefore $x_i(n) = y\xi_i$ for $n \in A$ sufficiently large. 
Thus $y\xi_i \in \Z$ for all $i$, hence $yL_n(\xi)\in \Z$ for all $n$. 
Since $L_n(\xi) \ne 0$, then we would have $|L_n(\xi)| \ge 1/|y|$ for all $n$, which contradicts the fact that $|L_n (\underline{\xi})|$ tends to zero.
Therefore  $\lim_{n \rightarrow \infty}|x_0(n)| =\infty $.

We fix $n$ sufficiently large and we write $\underline{x}$ instead of $\underline{x}(n)$. 
Let $k$ denote the least positive integer such that
$$
|x_0|\le \frac{1}{2|L_k(\underline{\xi})|}\cdotp
$$
Since $|x_0|$ tends to infinity with $n$, it follows that  $k$  also tends to infinity with $n$.
Moreover we have $ k\le n$ and
$$
|x_0|> \frac{1}{2|L_{k-1}(\underline{\xi})|}\cdotp
$$
Now we can write
\begin{equation}\label{Equation:1}
\sum_{i=0}^{r_k} \ell_{ik}x_i =x_0 \sum_{i=0}^{r_k} \ell_{ik} \xi_i+
\sum_{i=0}^{r_k}  \ell_{ik} (x_i-x_0\xi_i).
\end{equation}
%Write $\underline{\xi}$ for $(\xi_1,\ldots,\xi_{r_k})\in\R^{r_k}$. 
The term on the left-hand side of (\ref{Equation:1}) is an integer. On the right-hand side, the first term has absolute value equal to $|x_0 L_k(\underline{\xi})|$. Therefore it is  bounded above by $1/2$. We now use the following remark
(compare with \cite{FischlerZudilin}): {\it if an integer can be written as a sum $x+y$ of two real numbers with $|x|\le 1/2$, then $|y|\ge |x|$.} Hence
\begin{equation}\label{Equation:2}
\left|
\sum_{i=0}^{r_k}  \ell_{ik} (x_i-x_0\xi_i)\right|\ge 
\left|
x_0 \sum_{i=0}^{r_k} \ell_{ik} \xi_i
\right|.
\end{equation}
The term on the left-hand side of (\ref{Equation:2}) is bounded above by
$$
\sum_{i=0}^{r_k} 
|\ell_{ik} | \max_{1\le i\le r_n}| x_i-x_0\xi_i|\le 
Q_k  (2| L_n(\underline{\xi})|)^{1/r_n}
\le
Q_k \left(\frac{2}{A_n} \right)^{1/r_n},
$$
while the term on the right-hand side of (\ref{Equation:2}) is bounded below by
$$
|x_0|\cdotp |L_k(\underline{\xi})|
\ge 
 \frac{1}{2 |L_{k-1}(\underline{\xi})|}
 \cdotp  |L_k(\underline{\xi})|
 \ge 
 \frac{1}{2B_k}\cdotp
$$
Thus
$$
A_n \le 2(2 B_k Q_k)^{r_n}.
$$
Since $B_k Q_k \le B_n Q_n$, this implies $A_n \le 2(2 B_n Q_n)^{r_n}$. 
\end{proof}
%%%%%%%%%%%%%%%%%%%%%%%%%%%%%%%%%%%%%%%%%%%%%%%%%%%%%%%%%%%%%%%%%%%%%%%%%%%%%%%%%%%%%%%%%%%%%%%%%%%%%%%%%%%%%%%%%%%
\begin{remark}
Theorem \ref{MainTh}  does not contain Theorem 6 of \cite{FischlerZudilin}, since the latter introduces refinements involving gcd's, but its proof relies on the same arguments.   
%\\
%2. Under the hypotheses of Theorem \ref{MainTh}, when $r_n$ tends to infinity with $n$,  it may happen that (\ref{1}) has a solution with all $\xi_i$ %rational numbers. Here is an example which was kindly communicated to the author by  B.~de Mathan.
%Set $\xi_0=1$ and $\xi_i= -2^{-i}$  for all $i \ge 1$. 
%For $n \ge 0$, set $L_n(X)=X_0+\cdots +X_n$ and $r_n=n$.
%Then $L_n(\underline{\xi}) = 2^{-n}$ for all $n \ge 0$.
%Thus the hypotheses are verified with $Q_n=n+1, A_n=2^n$ and $B_n=2$.
%Also for sufficiently large integers $n$, we have $|x_0 (n)|\le 2^{n-1}$ and
%$|x_0 (n) (-2^i)-x_i (n)|\le 2^{-1 + (1/n) } ~ (1\le i\le r_n)$.
\end{remark}
%%%%%%%%%%%%%%%%%%%%%%%%%%%%%%%%%%%%%%%%%%%%%%%%%%%%%%%%%%%%%%%%%%%%%%%%%%%%%%%%%%%%%%%%%%%%%%%%%%%%%%%%%%%%%%%%%%%
We deduce from Theorem \ref{MainTh} a slight refinement of Theorem \ref{Nest} (Nesterenko's linear independence criterion).
\smallskip
\smallskip
\smallskip
\begin{cor}  \label{CorNest}
Let $\tau_1,\tau_2$ be positive real numbers and $\sigma(n)$  a non--decreasing positive  function such that  $\lim_{n \rightarrow \infty} \sigma(n) = \infty$.
Let $\underline{\vartheta} = (\vartheta_1, \ldots, \vartheta_m) \in \R^{m}$. 
Assume that, for all sufficiently large integers $n$, there exists a linear form with integer coefficients in $m+1$ variables
\[
L_n(\underline{X})=\ell_{0n}X_0+\ell_{1n}X_1+\cdots+\ell_{m n}X_{m}
\]
which satisfies the conditions
\[
\sum_{i=0}^{m} |\ell_{in}| \le e^{\sigma(n)}
\quad\hbox{ and }\quad
e^{-(\tau_1 + \mathit{o}(1))\sigma(n)}\le |L_n(1, \underline{\vartheta})|\le e^{-(\tau_2  + \mathit{o}(1))\sigma(n+1)}.
\]
Then $\dim_{\Q} (\Q + \Q \vartheta_1 + \cdots + \Q \vartheta_m) \ge (1 + \tau_1) / (1+ \tau_1 - \tau_2)$.
\end{cor}

\begin{remark}
%1. The assumptions of  Corollary \ref{CorNest}  imply    $\tau_2\le \tau_1$ because $\sigma(n)$ is non--decreasing.
%Since the conclusion is $\ge$ and not $>$, the statement without $\mathit{o}(1)$ in the upper and lower estimates for $ |L_n(1, %\underline{\vartheta})|$ is equivalent to Corollary \ref{CorNest}. Our formulation is more convenient to deal with; for instance the assumptions may %hold with $\tau_1=\tau_2$ since we include $\mathit{o}(1)$. 
%\\
1.  If $\ell $ is an integer and if the assumptions of  Corollary \ref{CorNest}  are satisfied for parameters $\tau_1$ and $\tau_2$ 
such that $\tau_1>(\ell-1)+(\tau_2-\tau_1)\ell$, 
then the conclusion yields
$\dim_{\Q} (\Q + \Q \vartheta_1 + \cdots + \Q \vartheta_m) \ge \ell+1$. In particular,  if $\tau_1 > (m-1) + m(\tau_1 - \tau_2)$, then $1, \vartheta_1, \ldots, \vartheta_m$ are linearly independent over $\Q$.
\\
2. 
The hypotheses of  Corollary \ref{CorNest}  imply that $L_n(1,\underline{\vartheta})\not=0$ and $\lim_{n\rightarrow \infty} L_n(1,\underline{\vartheta})=0$, therefore  
$\underline{\vartheta}
%(\vartheta_1, \ldots, \vartheta_m) 
\notin \Q^m$. 
When the aim is to prove $\underline{\vartheta} \notin \Q^m$, there is no need of a lower bound for $|L_n(1,\underline{\vartheta})|$. 
However, assuming a lower bound for $|L_n(1,\underline{\vartheta})|$ enables Nesterenko in \cite{MR783238} to reach a quantitative estimate.
%For instance, in the case $m=1$ the hypotheses of Theorem \ref{MainTh} imply that $\vartheta = \vartheta_1$ has an irrationality measure 
Consider for instance the special case $m=1$ with $\vartheta_1=\vartheta\in\R$: under the assumptions of Theorem \ref{Nest}, 
%there is no need of a lower bound for $|L_(1,\vartheta)|$ (only that it is not $0$) when the aim is to prove the irrationality of $\vartheta$. However, 
%assuming a lower bound for $|L_n(1,\vartheta)|$ enables one to prove that 
$\vartheta$ is not a Liouville number.
Conversely,  Theorem 1 of \cite{FischlerRivoal} shows that for a real number  $\vartheta=\vartheta_1$ which  is not a Liouville number, the assumptions of   Corollary \ref{CorNest} with $m=1$ are satisfied for suitable values of the parameters. 

\end{remark}

%%%%%%%%%%%%%%%%%%%%%%%%%%%%%%%%%%%%%%%%%%%%%%%%%%%%%%%%%%%%%%%%%%%%%%%%%%%%%%%%%%%%%%%%%%%%%%%%%%%%%%%%%%%%%%%%%%%
\begin{proof}[Proof of Corollary \ref{CorNest}]
Set $\vartheta_0 =1$. Let $\{\xi_0, \xi_1,\ldots, \xi_r\}$ with $\xi_0 =1$ be a basis of the vector space spanned by $\vartheta_0, \ldots, \vartheta_m$ over $\Q$.
Let $\underline{\xi} = (\xi_1, \ldots, \xi_r) \in \R^{r}$. 
Then for $0 \le i \le m$, there are integers $d , c_{i0}, \ldots, c_{ir}$ with $d >0$ such that
$d \vartheta_i = \displaystyle\sum_{j=0}^{r} c_{ij} \xi_j$.
For any $n \in \N$, define a linear form in $r+1$ variables $Y_0, \ldots, Y_r $ with integer coefficients by 
\begin{align*}
\Lambda_n (Y_0, \ldots, Y_r) &= L_n \left( \displaystyle\sum_{j=0}^{r} c_{0j} Y_j, \ldots, \displaystyle\sum_{j=0}^{r} c_{mj} Y_j \right) \\
&= \displaystyle\sum_{i=0}^{m} c_{i0} \ell_{in} Y_0 + \cdots + \displaystyle\sum_{i=0}^{m} c_{ir}\ell_{in} Y_r.
\end{align*}
From the definition of $\Lambda_n$ we infer $\Lambda_n (1, \underline{\xi}) = d \cdot L_n(1, \underline{\vartheta})$. 
We apply Theorem \ref{MainTh} with a finite sequence $\xi_0, \xi_1, \ldots, \xi_r$ with $r_n=r$  for all $n$. Let $\tau'_1$ and $\tau'_2$ satisfy $\tau'_1>\tau_1$ and $\tau'_2<\tau_2$. 
We take
$$
Q_n =   c e^{\sigma(n)},\quad A_n = e^{\tau'_2 \sigma(n+1)}
\quad\hbox{ and }\quad 
B_n = e^{(\tau'_1 -\tau'_2)\sigma(n)},
$$
where $c > 0$ is a suitable constant. 
We obtain
$$
e^{\tau'_2  \sigma(n+1)}  \le 2^{r+1} c^r e^{r(\tau'_1 - \tau'_2 +1) \sigma(n)}.
$$  
Since $\sigma(n)$ is non--decreasing and tends to infinity, this estimate implies
%$$
%\tau'_2 \sigma(n) \le r  (1+\tau_1 - \tau_2) \cdot \sigma(n),
%$$  
%hence
$$
 \tau'_2 \le r (1+\tau'_1 - \tau'_2).
 $$
Since this last inequality  holds for all $(\tau'_1,\tau'_2)$ with $\tau'_1>\tau_1$ and $\tau'_2<\tau_2$, we deduce 
$r \ge \tau_2 / (1+\tau_1 - \tau_2)$. 
Therefore 
$$
\dim_{\Q} (\Q + \Q \vartheta_1 + \cdots + \Q \vartheta_m) = r+1 \ge
\frac{ \tau_2 }{1+\tau_1 - \tau_2} +1 = \frac{1 + \tau_1}{1+ \tau_1 - \tau_2}\cdotp
$$
\end{proof}
%%%%%%%%%%%%%%%%%%%%%%%%%%%%%%%%%%%%%%%%%%%%%%%%%%%%%%%%%%%%%%%%%%%%%%%%%%%%%%%%%%%%%%%%%%%%%%%%%%%%%%%%%%%%%%%%%%%

%%%%%%%%%%%%%%%%%%%%%%%%%%%%%%%%%%%%%%%%%%%%%%%%%%%%%%%%%%%%%%%%%%%%%%%%%%%%%%%%%%%%%%%%%%%%%%%%%%%%%%%%%%%%%%%%%%%
%%%%%%%%%%%%%%%%%%%%%%%%%%%%%%%%%%%%%%%%%%%%%%%%%%%%%%%%%%%%%%%%%%%%%%%%%%%%%%%%%%%%%%%%%%%%%%%%%%%%%%%%%%%%%%%%%%%
\section{Criteria for algebraic independence}
%%%%%%%%%%%%%%%%%%%%%%%%%%%%%%%%%%%%%%%%%%%%%%%%%%%%%%%%%%%%%%%%%%%%%%%%%%%%%%%%%%%%%%%%%%%%%%%%%%%%%%%%%%%%%%%%%%%
%%%%%%%%%%%%%%%%%%%%%%%%%%%%%%%%%%%%%%%%%%%%%%%%%%%%%%%%%%%%%%%%%%%%%%%%%%%%%%%%%%%%%%%%%%%%%%%%%%%%%%%%%%%%%%%%%%%

We deduce from Theorem \ref{MainTh} the following criterion for algebraic independence:
\smallskip
\smallskip
\smallskip 
\begin{cor}  \label{CorAlg}
Let $\vartheta_1, \ldots, \vartheta_t$ be real numbers, $(d_n)_{n \ge 0}$  a sequence of positive integers, $(\alpha_n)_{n \ge 0}, (\beta_n)_{n \ge 0}$ and $(\gamma_n)_{n \ge 0}$  sequences of positive real numbers such that $\lim_{n \rightarrow \infty} d_n^{-t} \gamma_n = \infty$ and, for all sufficiently large integers $n$, 
$$
\alpha_n + \beta_n - \gamma_{n-1} \le \alpha_{n+1} + \beta_{n+1} - \gamma_{n}.
$$
Assume that there exists  a sequence $(P_n)_{n \ge 0}$ of polynomials in $\Z[X_1, \ldots, X_t]$, where $P_n$ has total degree at most $d_n$ and length at most $e^{\beta_n}$, satisfying
$$
e^{-\alpha_n} \le |P_n(\vartheta_1, \ldots, \vartheta_t)| \le e^{-\gamma_n}.
$$
Then 
$$
\gamma_n \le \log{2} + \left( \dbinom{d_n +t}{t} -1 \right) (\alpha_n + \beta_n - \gamma_{n-1} + \log{2})
$$ 
for all sufficiently large integers $n$.
\end{cor}
%%%%%%%%%%%%%%%%%%%%%%%%%%%%%%%%%%%%%%%%%%%%%%%%%%%%%%%%%%%%%%%%%%%%%%%%%%%%%%%%%%%%%%%%%%%%%%%%%%%%%%%%%%%%%%%%%%%
\begin{proof}
The number of integer  tuples $\underline{i}=(i_1, \ldots,i_t)$ in $\Z_{\ge 0}^t$ 
with $i_1+\cdots+i_t\le d_n$ 
is $\dbinom{d_n +t}{t}$. 
For each of these $\underline{i}$, let $\xi_{\underline{i}} = \vartheta_1^{i_1} \ldots \vartheta_t^{i_t}$. 
Let $n \in \N$. We write 
$$
P_n(X_1, \ldots, X_t) = \sum_{\underline{i}} a_{\underline{i},n} X_1^{i_1} \ldots X_t^{i_t}.
$$
Put $L_n(\underline{X})= \sum_{\underline{i}} a_{\underline{i},n} X_{\underline{i}}$.
Then $L_n(\underline{\xi}) = P_n(\vartheta_1, \ldots, \vartheta_t)$.
Applying Theorem \ref{MainTh} with   $r_n = \dbinom{d_n +t}{t} -1$, $Q_n = e^{\beta_n}$, $A_n = e^{\gamma_n}$ and $B_n = e^{\alpha_n - \gamma_{n-1}}$, 
we obtain 
\[
\gamma_n \le \log{2} + \left( \dbinom{d_n +t}{t} -1 \right) (\alpha_n + \beta_n - \gamma_{n-1} + \log{2})
\]
 for all sufficiently large integers $n$.
\end{proof}
%%%%%%%%%%%%%%%%%%%%%%%%%%%%%%%%%%%%%%%%%%%%%%%%%%%%%%%%%%%%%%%%%%%%%%%%%%%%%%%%%%%%%%%%%%%%%%%%%%%%%%%%%%%%%%%%%%%
We deduce from Corollary \ref{CorAlg} the following special cases. In the first one, we consider a sequence of bounded degree polynomials. Since our method relies on linear elimination, our results are sharp when the dimension of the space is not too large. As far as criteria for algebraic independence are concerned, we get sharper results when the degree is bounded. 
%%%%%%%%%%%%%%%%%%%%%%%%%%%%%%%%%%%%%%%%%%%%%%%%%%%%%%%%%%%%%%%%%%%%%%%%%%%%%%%%%%%%%%%%%%%%%%%%%%%%%%%%%%%%%%%%%%%
\begin{cor} \label{CorAlg0}
Let $\vartheta_1, \ldots, \vartheta_t$ be real numbers, $d$  a positive integer, $\tau_1,\tau_2$  positive real numbers and $\sigma(n)$  a non--decreasing positive  function such that  $\lim_{n \rightarrow \infty} \sigma(n) = \infty$.
Assume that there exists  a sequence $(P_n)_{n \ge 0}$ of polynomials  in $\Z[X_1, \ldots, X_t]$, where $P_n$ has total degree at most $d$ and length at most $e^{\sigma(n)}$, satisfying 
$$
e^{-(\tau_1 + \mathit{o}(1))\sigma(n)}\le |P_n(\vartheta_1, \ldots, \vartheta_t)| \le e^{-(\tau_2  + \mathit{o}(1))\sigma(n+1)}.
$$
Then 
$$
\tau_2 \le \left( \dbinom{d +t}{t} - 1 \right) (1+\tau_1 - \tau_2).
$$
\end{cor}
%Notice that if $\tau_1 = \tau_2 = \tau$, then the conclusion is  $\tau \le \dbinom{d +t}{t} - 1$, where the right hand side is the Dirichlet's exponent.
%%%%%%%%%%%%%%%%%%%%%%%%%%%%%%%%%%%%%%%%%%%%%%%%%%%%%%%%%%%%%%%%%%%%%%%%%%%%%%%%%%%%%%%%%%%%%%%%%%%%%%%%%%%%%%%%%%%
\begin{proof} 
Let $\varepsilon$ be a positive real number $< \tau_2$. Set $\tau'_1=\tau_1+\varepsilon$ and $\tau'_2=\tau_2-\varepsilon$. 
We apply Corollary \ref{CorAlg} with $d_n =d$, $\alpha_n =\tau'_1 \sigma(n)$, $\beta_n = \sigma(n)$ and $\gamma_n =\tau'_2 \sigma(n+1)$. The conclusion follows by letting $\varepsilon$ tend to $0$.
\end{proof}
%%%%%%%%%%%%%%%%%%%%%%%%%%%%%%%%%%%%%%%%%%%%%%%%%%%%%%%%%%%%%%%%%%%%%%%%%%%%%%%%%%%%%%%%%%%%%%%%%%%%%%%%%%%%%%%%%%%
Here are examples with  sequences of polynomials whose degrees could tend to infinity.
\begin{cor} \label{CorAlg1}
Let $\vartheta_1, \ldots, \vartheta_t$ be real numbers, $\beta$, $\delta$, $\kappa$ and $\lambda$  positive real numbers, $\alpha$ and $\gamma$  real numbers. Assume  $\kappa \le 1/t$.
Further, assume that  there exists  a sequence $(P_n)_{n \ge 0}$ of polynomials in $\Z[X_1, \ldots, X_t]$, where $P_n$ has total degree at most $\delta n^{\kappa}+ \mathit{o}(n^{\kappa})$ and length at most $e^{\beta n + \mathit{o}(n)}$, satisfying
$$
e^{- \lambda n^{1 + t \kappa} -  (\alpha   + \mathit{o}(1))n} \le |P_n(\vartheta_1, \ldots, \vartheta_t)| \le e^{- \lambda n^{1 + t \kappa} - (\gamma   + \mathit{o}(1))n}.
$$
\\
Then 
$$ 
\lambda (1 -2 \delta^t/ t!) \le \delta^t(\alpha + \beta -\gamma)/t! \quad\hbox{ when }\quad \kappa = 1/t,
$$
and 
$$
\lambda  \le \delta^t(\alpha + \beta -\gamma)/t!\quad\hbox{ when }\quad \kappa < 1/t.
$$ 
\end{cor}
%%%%%%%%%%%%%%%%%%%%%%%%%%%%%%%%%%%%%%%%%%%%%%%%%%%%%%%%%%%%%%%%%%%%%%%%%%%%%%%%%%%%%%%%%%%%%%%%%%%%%%%%%%%%%%%%%%%
\begin{proof} 
Let $\varepsilon$ be a positive real number. Set $\alpha'=\alpha+\varepsilon$,  $\beta'=\beta+\varepsilon$,  $\gamma'=\gamma - \varepsilon$ and $\delta'=\delta+\varepsilon$.
Let $d_n = \delta' n^{\kappa}$, $\beta_n = \beta' n$, $\gamma_n =\lambda n^{1 + t \kappa} + \gamma' n$,
$$
\alpha_n =
\begin{cases}
 \lambda n^2 + \alpha' n & \quad\hbox{ when }\quad \kappa = 1/t,
\\
 \lambda (n-1)^{1 + t \kappa} + \alpha' n & \quad\hbox{ when }\quad \kappa < 1/t.
 \end{cases}
$$ 
We obtain $e^{-\alpha_n} \le |P_n(\vartheta_1, \ldots, \vartheta_t)| \le e^{-\gamma_n}$ for $n$ sufficiently large.
Besides,
$$
\alpha_n+\beta_n-\gamma_{n-1} =
\begin{cases}
 (2 \lambda + \alpha'+\beta'-\gamma')n -\lambda + \gamma' & \quad\hbox{ when }\quad \kappa = 1/t,
 \\
 (\alpha'+\beta'-\gamma')n+\gamma'  &\quad\hbox{ when }\quad \kappa < 1/t.
 \end{cases}
$$
Since the hypothesis implies that $\alpha -\gamma \ge 0$, it follows that the sequence 
$$
(\alpha_n+\beta_n-\gamma_{n-1})_{n\ge 1}
$$
is non--decreasing
for sufficiently large $n$.
The conclusion follows by applying Corollary \ref{CorAlg} and letting $\varepsilon$ tend to $0$.
\end{proof}
%%%%%%%%%%%%%%%%%%%%%%%%%%%%%%%%%%%%%%%%%%%%%%%%%%%%%%%%%%%%%%%%%%%%%%%%%%%%%%%%%%%%%%%%%%%%%%%%%%%%%%%%%%%%%%%%%%%

%Here is a further example where we do not assume the degree to be bounded.
\begin{cor} 
Let $\vartheta_1, \ldots, \vartheta_t$ be real numbers, $(d_n)_{n \ge 0}$ a sequence of positive integers and $\beta$, $\kappa$ and $\lambda$  positive real numbers.  Assume $$
d_n = \mathit{o}(n^{1/t})
\quad\hbox{and}\quad d_n^t - d_{n-1}^t \le \kappa 
$$ 
for all sufficiently large integers $n$.
Further, assume that  there exists  a sequence $(P_n)_{n \ge 0}$ of polynomials in $\Z[X_1, \ldots, X_t]$,  where $P_n$ has total degree at most $d_n$ and length at most $e^{\beta n + \mathit{o}(n)}$ satisfying
$$
|P_n(\vartheta_1, \ldots, \vartheta_t)| = e^{- (\lambda n d_n^t + \mathit{o}(n))}.
$$
\\
Then 
$$
\lambda (t!-\kappa) \le \beta.
$$
\end{cor}
%%%%%%%%%%%%%%%%%%%%%%%%%%%%%%%%%%%%%%%%%%%%%%%%%%%%%%%%%%%%%%%%%%%%%%%%%%%%%%%%%%%%%%%%%%%%%%%%%%%%%%%%%%%%%%%%%%%
\begin{proof}  
Let $\varepsilon>0$, $\alpha_n =\lambda (n-1) (d_{n-1}^t + \kappa ) +\varepsilon n$, $\gamma_n =\lambda n d_n^t - \varepsilon n$ and $\beta_n = \beta n + \varepsilon n$.
For $n$ sufficiently large, since
$$
\lambda n (d_n^t - d_{n-1}^t) + \lambda (d_{n-1}^t + \kappa) + \mathit{o}(n)\le \lambda n \kappa + \mathit{o}(n) \le \lambda n \kappa + \varepsilon n,
$$
we have 
$$
\lambda n d_n^t + \mathit{o}(n) \le \lambda (n-1) (d_{n-1}^t + \kappa ) +\varepsilon n.
$$
It follows that
$|P_n(\vartheta_1, \ldots, \vartheta_t)| \ge e^{-\alpha_n}$.  

Also, $|P_n(\vartheta_1, \ldots, \vartheta_t)| \le e^{-\gamma_n}$ for $n$ sufficiently large.

Since $\alpha_n+\beta_n-\gamma_{n-1} = (\lambda \kappa + \beta + 3\varepsilon) n - \lambda \kappa - \varepsilon$, it follows that the sequence $(\alpha_n+\beta_n-\gamma_{n-1})_{n\ge 1}$ is non--decreasing
for sufficiently large $n$.
The conclusion follows by applying Corollary \ref{CorAlg} and letting $\varepsilon$ tend to $0$.
\end{proof}
%%%%%%%%%%%%%%%%%%%%%%%%%%%%%%%%%%%%%%%%%%%%%%%%%%%%%%%%%%%%%%%%%%%%%%%%%%%%%%%%%%%%%%%%%%%%%%%%%%%%%%%%%%%%%%%%%%%

%%%%%%%%%%%%%%%%%%%%%%%%%%%%%%%%%%%%%%%%%%%%%%%%%%%%%%%%%%%%%%%%%%%%%%%%%%%%%%%%%%%%%%%%%%%%%%%%%%%%%%%%%%%%%%%%%%%

Algebraic independence results follow from our criteria: 

%%%%%%%%%%%%%%%%%%%%%%%%%%%%%%%%%%%%%%%%%%%%%%%%%%%%%%%%%%%%%%%%%%%%%%%%%%%%%%%%%%%%%%%%%%%%%%%%%%%%%%%%%%%%%%%%%%%
\begin{cor} \label{CorDepAlg}
Let $\vartheta_1,\ldots, \vartheta_t$ be real numbers,  $d$, $\delta$ two positive
integers with $d\ge \delta$   and $\tau$,  $\eta$ two  positive real numbers  satisfying
$$
\frac{\tau}{ 1+\eta}>
\binom{t+d}{t}-\binom{t+d-\delta}{t}-1.
$$
Let $\bigl(\sigma(n)\bigr)_{n\ge 1}$ be a non--decreasing  sequence of real
numbers which tends to infinity.
Assume that   there is a sequence $(P_n)_{n\ge n_0}$ of polynomials in
$\Z[X_1,\ldots,X_t]$, where $P_n$ has total degree $\le d$ and length $\le e^{\sigma(n)}$, such
that
$$
e^{-(\tau +\eta+\mathit{o}(1))\sigma(n)}
\le |P_n(\vartheta_1,\ldots,\vartheta_t)|
\le e^{-(\tau+\mathit{o}(1)) \sigma(n+1)}.
$$
Then $\vartheta_1,\ldots,\vartheta_t$ do not satisfy any algebraic dependence relation
with rational coefficients of degree $\le \delta$.
\end{cor}
%%%%%%%%%%%%%%%%%%%%%%%%%%%%%%%%%%%%%%%%%%%%%%%%%%%%%%%%%%%%%%%%%%%%%%%%%%%%%%%%%%%%%%%%%%%%%%%%%%%%%%%%%%%%%%%%%%%
\begin{proof} 
We use Corollary \ref{CorNest} with $\tau_1=\tau+\eta$, $\tau_2=\tau$, and $1,\vartheta_1,\ldots,\vartheta_m$ in  Corollary \ref{CorNest} replaced by 
$\vartheta_1^{i_1}\cdots\vartheta_t^{i_t}$ with $i_1+\cdots+i_t \le d$. 
Then the dimension of the subspace of $\R$ over $\Q$ spanned by $\{ \vartheta_1^{i_1}\cdots\vartheta_t^{i_t} ~ | ~  i_1+\cdots+i_t\le d  \}$ is bounded below by
$$
   1+\frac{\tau}{1+\eta}\cdotp
$$

Now, we suppose that there exists a polynomial $Q \in \Z[X_1,\ldots,X_t]$ of total degree $\le \delta$ such that $Q(\vartheta_1,\ldots,\vartheta_t) = 0$.
Let $\psi$ be the linear transformation from the set of all polynomials of total degree $\le d$ in $\Q[X_1 , \ldots , X_t]$ to $\R$ defined by 
$\psi (P(X_1, \ldots , X_t)) =  P(\vartheta_1,\ldots,\vartheta_t)$ for $P(X_1, \ldots , X_t) \in \Q[X_1 , \ldots , X_t]$. 
Then
$$
\{ X_1^{j_1} \cdots X_t^{j_t} \cdot Q ~ | ~  j_1+\cdots+j_t\le d -\delta \} \subset \ker{\psi}.
$$
It follows that 
$$
 \dim (\ker{\psi}) \ge   \dbinom{t+d- \delta}{t}. 
 $$
Therefore,  the dimension of the image of $\psi$, which is equal to
$$
\dbinom{t+d}{t} - \dim (\ker{\psi}),
$$
is bounded above by 
$$
  \dbinom{t+d}{t} - \dbinom{t+d- \delta}{t}. 
$$
Since the image of $\psi$ is the subspace of $\R$ over $\Q$ spanned by 
$$
\{ \vartheta_1^{i_1}\cdots\vartheta_t^{i_t} ~ | ~  i_1+\cdots+i_t\le d  \},
$$ 
we have
$$
1 + \frac{\tau}{1+\eta} \le \dbinom{t+d}{t} - \dbinom{t+d- \delta}{t} ,
$$
which is a contradiction.
\end{proof}
%%%%%%%%%%%%%%%%%%%%%%%%%%%%%%%%%%%%%%%%%%%%%%%%%%%%%%%%%%%%%%%%%%%%%%%%%%%%%%%%%%%%%%%%%%%%%%%%%%%%%%%%%%%%%%%%%%%
When the assumptions of Corollary \ref{CorDepAlg} are satisfied for any $\delta>0$, we deduce the algebraic independence of the numbers $\vartheta_1,\ldots,\vartheta_t$.
%%%%%%%%%%%%%%%%%%%%%%%%%%%%%%%%%%%%%%%%%%%%%%%%%%%%%%%%%%%%%%%%%%%%%%%%%%%%%%%%%%%%%%%%%%%%%%%%%%%%%%%%%%%%%%%%%%%
\begin{cor}
Let $\vartheta_1,\ldots, \vartheta_t$ be real numbers and $(\tau_d)_{d\ge 1}$, 
$(\eta_d)_{d\ge 1}$ two sequences of positive real numbers satisfying
$$
\frac{\tau_d}{d^{t-1}(1+\eta_d)}\longrightarrow+\infty.
$$
Further, let $\bigl(\sigma(n)\bigr)_{n\ge 1}$ be a non--decreasing  sequence of real
numbers which tends to infinity.
Assume that for all sufficiently large $d$, there is a sequence $(P_n)_{n\ge n_0(d)}$ of
polynomials in $\Z[X_1,\ldots,X_t]$, where $P_n$ has total degree $\le d$ and length $\le e^{\sigma(n)}$,
such that, for $n\ge n_0(d)$,
$$
e^{-(\tau_d+\eta_d)\sigma(n)}
\le |P_n(\vartheta_1,\ldots,\vartheta_t)|
\le e^{-\tau_d
\sigma(n+1)}.
$$
Then $\vartheta_1,\ldots,\vartheta_t$ are algebraically independent.
\end{cor}
%%%%%%%%%%%%%%%%%%%%%%%%%%%%%%%%%%%%%%%%%%%%%%%%%%%%%%%%%%%%%%%%%%%%%%%%%%%%%%%%%%%%%%%%%%%%%%%%%%%%%%%%%%%%%%%%%%%
\begin{proof} 
Let $\delta$ be a positive integer. For sufficiently large $d$, we have
$$
\dbinom{t+d}{t} - \dbinom{t+d- \delta}{t} -1 < c d^{t-1} < \frac{\tau_d}{1 + \eta_d}
$$
where $c>0$ is a suitable constant.
Then the hypotheses of Corollary \ref{CorDepAlg} are satisfied with $\tau = \tau_d$ and $\eta = \eta_d$.
Hence $\vartheta_1,\ldots,\vartheta_t$ do not satisfy any algebraic dependence relation with rational coefficients of degree $\le\delta$. 
This is true for all $\delta$.
Therefore $\vartheta_1,\ldots,\vartheta_t$ are algebraically independent.
\end{proof}
%%%%%%%%%%%%%%%%%%%%%%%%%%%%%%%%%%%%%%%%%%%%%%%%%%%%%%%%%%%%%%%%%%%%%%%%%%%%%%%%%%%%%%%%%%%%%%%%%%%%%%%%%%%%%%%%%%%

%%%%%%%%%%%%%%%%%%%%%%%%%%%%%%%%%%%%%%%%%%%%%%%%%%%%%%%%%%%%%%%%%%%%%%%%%%%%%%%%%%%%%%%%%%%%%%%%%%%%%%%%%%%%%%%%%%%
\section{Transcendence criteria}
%%%%%%%%%%%%%%%%%%%%%%%%%%%%%%%%%%%%%%%%%%%%%%%%%%%%%%%%%%%%%%%%%%%%%%%%%%%%%%%%%%%%%%%%%%%%%%%%%%%%%%%%%%%%%%%%%%%
%%%%%%%%%%%%%%%%%%%%%%%%%%%%%%%%%%%%%%%%%%%%%%%%%%%%%%%%%%%%%%%%%%%%%%%%%%%%%%%%%%%%%%%%%%%%%%%%%%%%%%%%%%%%%%%%%%%

The special case $t=1$ of Corollary \ref{CorAlg} yields a criterion for transcendence, which is similar to a well known result due to A.O.~Gel'fond  (see for instance Lemma 6.3, \S~1.3, Chap.~6 of \cite{MR1603608}).  

\begin{cor}  \label{CorTcd1}
Let $\vartheta$ be a real number, $(d_n)_{n \ge 0}$  a sequence of positive integers, $(\alpha_n)_{n \ge 0},(\beta_n)_{n \ge 0}$ and $(\gamma_n)_{n \ge 0}$  sequences of positive real numbers such that $\lim_{n \rightarrow \infty}\gamma_n /d_n= \infty$ and, for all sufficiently large integers $n$, $\alpha_n + \beta_n - \gamma_{n-1} \le \alpha_{n+1} + \beta_{n+1} - \gamma_{n}$.
Assume that   there exists  a sequence $(P_n)_{n \ge 0}$ of polynomials in $\Z[X]$, where $P_n$ has  degree at most $d_n$ and length at most $e^{\beta_n}$, satisfying 
$$
e^{-\alpha_n} \le |P_n(\vartheta)| \le e^{-\gamma_n}.
$$
Then 
$$
\gamma_n \le \log{2} + d_n (\alpha_n + \beta_n - \gamma_{n-1} + \log{2})
$$
for all sufficiently large integers $n$.
\end{cor}
%%%%%%%%%%%%%%%%%%%%%%%%%%%%%%%%%%%%%%%%%%%%%%%%%%%%%%%%%%%%%%%%%%%%%%%%%%%%%%%%%%%%%%%%%%%%%%%%%%%%%%%%%%%%%%%%%%%

Here is the special case $t=1$ of Corollary \ref{CorAlg0}. 
\begin{cor}  \label{CorTcd2} 
Let $\vartheta$ be a real number, $d$  a positive integer,  $\tau_1,\tau_2$ positive real numbers and $\sigma(n)$  a non--decreasing positive  function such that  $\lim_{n \rightarrow \infty} \sigma(n) = \infty$.
Assume that there exists  a sequence $(P_n)_{n \ge 0}$ of polynomials  in $\Z[X]$, where $P_n$ has degree at most $d$ and length at most $e^{\sigma(n)}$, satisfying 
$$
e^{-(\tau_1 + \mathit{o}(1))\sigma(n)}\le  |P_n(\vartheta)|  \le e^{-(\tau_2  + \mathit{o}(1))\sigma(n+1)}.
$$
Then 
$$
\tau_2 \le d+d(\tau_1 - \tau_2).
$$
\end{cor}

Note that if $\tau_1 = \tau_2$, then the conclusion is $\tau_1 \le d$. 

%%%%%%%%%%%%%%%%%%%%%%%%%%%%%%%%%%%%%%%%%%%%%%%%%%%%%%%%%%%%%%%%%%%%%%%%%%%%%%%%%%%%%%%%%%%%%%%%%%%%%%%%%%%%%%%%%%%
It is interesting to compare Corollary \ref{CorTcd2} with the results following from the proof of Gel'fond's criterion: in our present paper,  we use only  linear elimination, while Gel'fond's proof relies on algebraic elimination. In Gel'fond's criterion, there is no need of a lower bound for $|P_n(\vartheta)|$, but the conclusion is not so sharp. For instance 
 Lemma 14 \cite{MR1475689} implies the following result:
\begin{prop}[\cite{MR1475689}]  \label{PropLaurentRoy}
Let $\vartheta$ be a complex number, $d$  a positive integer, $\alpha$ and $\beta$  positive real numbers.
Assume that  there exists  a sequence $(P_n)_{n \ge 0}$ of polynomials in $\Z[X]$, where $P_n$ has degree at most $d$ and length at most $e^{\beta n + \mathit{o}(n)}$, satisfying 
$0 <  |P_n(\vartheta)| \le e^{-\alpha n + \mathit{o}(n)}$.
Then $\alpha \le 3 d \beta$.
\end{prop}
%%%%%%%%%%%%%%%%%%%%%%%%%%%%%%%%%%%%%%%%%%%%%%%%%%%%%%%%%%%%%%%%%%%%%%%%%%%%%%%%%%%%%%%%%%%%%%%%%%%%%%%%%%%%%%%%%%%
Under the assumptions of Proposition \ref{PropLaurentRoy}, if we assume that $\vartheta$ is a real number and
$$
 |P_n(\vartheta)| = e^{-\alpha n + \mathit{o}(n)},
$$ 
then Corollary \ref{CorTcd2} gives the stronger conclusion $\alpha \le  d \beta$.

We conclude with the special case $t=\kappa=1$ of  Corollary \ref{CorAlg1},  which is also the special case where $\delta_n$, $\alpha_n$, $\beta_n$ and $\gamma_n$ satisfy
 $$
 \delta_n = \delta n + \mathit{o}(n), \;  \alpha_n = \lambda n^2 + \alpha n + \mathit{o}(n), \;  \beta_n = \beta n + \mathit{o}(n) \; \hbox{ and } \;  \gamma_n =\lambda n^2 +\gamma n + \mathit{o}(n)
 $$
 of Corollary \ref{CorTcd1}. 
%%%%%%%%%%%%%%%%%%%%%%%%%%%%%%%%%%%%%%%%%%%%%%%%%%%%%%%%%%%%%%%%%%%%%%%%%%%%%%%%%%%%%%%%%%%%%%%%%%%%%%%%%%%%%%%%%%%
\smallskip
\smallskip
\smallskip
\begin{cor} \label{CorTcd3}
Let $\vartheta$ be a real number, $\beta$, $\delta$ and $\lambda$  positive real numbers  such that $\delta < 1/2$, $\alpha$ and $\gamma$  real numbers.
Assume that  there exists  a sequence $(P_n)_{n \ge 0}$ of polynomials in $\Z[X]$, where $P_n$ has degree at most $\delta n+ \mathit{o}(n)$ and length at most $e^{\beta n + \mathit{o}(n)}$, satisfying 
$$
e^{- (\lambda n^2 + \alpha n + \mathit{o}(n))} \le |P_n(\vartheta)| \le e^{- (\lambda n^2 + \gamma n + \mathit{o}(n))}.
$$
Then $\lambda(1-2\delta)  \le \delta( \alpha + \beta -\gamma)$.
\end{cor}
%%%%%%%%%%%%%%%%%%%%%%%%%%%%%%%%%%%%%%%%%%%%%%%%%%%%%%%%%%%%%%%%%%%%%%%%%%%%%%%%%%%%%%%%%%%%%%%%%%%%%%%%%%%%%%%%%%%
 
%%%%%%%%%%%%%%%%%%%%%%%%%%%%%%%%%%%%%%%%%%%%%%%%%%%%%%%%%%%%%%%%%%%%%%%%%%%%%%%%%%%%%%%%%%%%%%%%%%%%%%%%%%%%%%%%%%%

%%%%%%%%%%%%%%%%%%%%%%%%%%
\bigskip

\noindent{\bf Acknowledgements:\ } The author thanks Michel Waldschmidt for his support and directions, also Bernard de Mathan, Francesco Amoroso St\'{e}phane Fischler and W. Dale Brownawell for their helpful remarks.

%%%%%%%%%%%%%%%%%%%%%%%%%%

\vspace{0.6cm}

\begin{small}
\noindent
Amarisa Chantanasiri  \\
Institut de Math\'{e}matiques de Jussieu, \\
Th\'{e}orie des Nombres 7C20, \\
175 rue du Chevaleret 75013 Paris, France \\
Email: \texttt{chantanasiri@math.jussieu.fr}
\end{small}

\label{last-page}

\end{document}